\documentclass[11pt]{amsart}%
\usepackage{graphicx}
\usepackage{amscd, color}
\usepackage{amsmath}
\usepackage{amsfonts}
\usepackage{amssymb}%
\providecommand{\U}[1]{\protect\rule{.1in}{.1in}}
\providecommand{\U}[1]{\protect\rule{.1in}{.1in}}
\providecommand{\U}[1]{\protect\rule{.1in}{.1in}} \textwidth 16.3cm
\theoremstyle{plain}

\newtheorem{theorem}{Theorem}[section]
\newtheorem{proposition}[theorem]{Proposition}
\newtheorem{corollary}[theorem]{Corollary}
\newtheorem{example}[theorem]{Example}

\newtheorem{lemma}[theorem]{Lemma}

\numberwithin{equation}{section}

\begin{document}
\title[Remarks on cotype absolutely summing multilinear operators]{Remarks on cotype and absolutely summing multilinear operators}
\author{A. Thiago Bernardino}
\address[A. Thiago Bernardino]{ UFRN/CERES - Centro de Ensino Superior do Serid\'{o},
Rua Joaquim Greg\'{o}rio, S/N, 59300-000, Caic\'{o}- RN, Brazil}
\email{thiagobernardino@yahoo.com.br}
\thanks{}

\begin{abstract}
In this short note we present some new results concerning cotype and
absolutely summing multilinear operators.

\end{abstract}
\maketitle

\section{Introduction}

In this note the letters $X_{1},...,X_{n},X,Y$ will denote Banach spaces over
the scalar field $\mathbb{K}=\mathbb{R}$ or $\mathbb{C}$.

From now on the space of all continuous $n$-linear operators from $X_{1}%
\times\cdots\times X_{n}$ to $Y$ will be denoted by $\mathcal{L}%
(X_{1},...,X_{n};Y).$ If $1\leq s<\infty$, the symbol $s^{\ast}$ represents
the conjugate of $s$. It will be convenient to adopt that $s/\infty=0$ for any
$s>0.$ For $1\leq q<\infty,$ denote by $\ell_{q}^{w}(X)$ the set
$\{(x_{j})_{j=1}^{\infty}\subset X:\sup_{\varphi\in B_{X^{\ast}}}%
{\textstyle\sum\nolimits_{j}}
\left\vert \varphi(x_{j})\right\vert ^{q}<\infty\}.$

If $0<p,q_{1},...,q_{n}<\infty$ and
\[
\frac{1}{p}\leq\frac{1}{q_{1}}+\cdots+\frac{1}{q_{n}},
\]
a multilinear operator $T\in\mathcal{L}(X_{1},...,X_{n};Y)$ is
absolutely\emph{ }$(p;q_{1},...,q_{n})$-summing if
\[
(T(x_{j}^{(1)},...,x_{j}^{(n)}))_{j=1}^{\infty}\in\ell_{p}(Y)
\]
for every $(x_{j}^{(k)})_{j=1}^{\infty}\in\ell_{q_{k}}^{w}(X_{k}),k=1,...,n.$
In this case we write $T\in\Pi_{p,q_{1},...,q_{n}}^{n}\left(  X_{1}%
,...,X_{n};Y\right)  $. If $q_{1}=\cdots=q_{n}=q,$ we write $\Pi_{p,q}%
^{n}\left(  X_{1},...,X_{n};Y\right)  $ instead of $\Pi_{p,q,...,q}^{n}\left(
X_{1},...,X_{n};Y\right)  .$ For details on the linear theory we refer to the
excellent monograph \cite{djt} and for the multilinear theory we refer to
\cite{am, CD, panor} and references therein.

This paper deals with the connection between cotype and absolutely summing
multilinear operators; this line of investigation begins with \cite{Bot} and
was followed by several recent papers (we refer, for example, to \cite{gsd,
pp, df, MT, danielstudia, dd, popa5, popa} and for a full panorama we mention
\cite{panor}). The following result appears in \cite[Theorem 3 and Remark
2]{junek} and \cite[Corollary 4.6]{popa} (see also \cite[Theorem 3.8
(ii)]{gsd} for a particular case):

\begin{theorem}
[Inclusion Theorem]\label{tt} Let $X_{1},...,X_{n}$ be Banach spaces with
cotype $s$ and $n\geq2$ be a positive integer:

(i) If $s=$ $2,$ then
\[
\Pi_{q;q}^{n}(X_{1},...,X_{n};Y)\subseteqq\Pi_{p;p}^{n}(X_{1},...,X_{n};Y)
\]
holds true for $1\leq p\leq q\leq2$.

(ii) If $s>2,$ then
\[
\Pi_{q;q}^{n}(X_{1},...,X_{n};Y)\subseteqq\Pi_{p;p}^{n}(X_{1},...,X_{n};Y)
\]
holds true for $1\leq p\leq q<s^{\ast}<2$.
\end{theorem}

As a consequence of results from \cite{Blasco} one can easily prove the
following generalization of this result (see \cite{th} for details):

\begin{theorem}
\label{klmm}If $X_{1}$ has cotype $2$ and $1\leq p\leq s\leq2,$ then
\[
\Pi_{(s;s,q,...,q)}^{n}(X_{1},...,X_{n};Z)\subseteqq\Pi_{(p;p,q,....,q)}%
^{n}(X_{1},...,X_{n};Z)
\]
for all $X_{2},...,X_{n}$,$Z$ and all $q\geq1.$ In particular%
\begin{equation}
\Pi_{s;s}^{n}(X_{1},...,X_{n};Z)\subseteqq\Pi_{(p;p,s,....,s)}^{n}%
(X_{1},...,X_{n};Z)\subseteqq\Pi_{p;p}^{n}(X_{1},...,X_{n};Z). \label{uuu}%
\end{equation}
A similar result, mutatis mutandis, holds if $X_{j}$ (instead of $X_{1}$) has
cotype $2.$
\end{theorem}

In this note we remark that analogous results hold for other situations in
which the spaces involved may have different cotypes and no space may have
necessarily cotype $2$.

\section{Results}

The following proposition can be found in \cite{gsd}:

\begin{proposition}
\label{prop}Let $1\leq p_{1},...,p_{n},p,q_{1},...,q_{n},q\leq\infty$ such
that $1/t\leq\sum_{j=1}^{n}1/t_{j}$ for $t\in\left\{  p,q\right\}  .$ Let
$0<\theta<1$ and define%
\[
\frac{1}{r}=\frac{1-\theta}{p}+\frac{\theta}{q}\text{ and }\frac{1}{r_{j}%
}=\frac{1-\theta}{p_{j}}+\frac{\theta}{q_{j}}\text{ for all }j=1,...,n
\]
and let $T\in\mathcal{L}\left(  X_{1},...,X_{n};Y\right)  .$ Then%
\[
T\in\Pi_{p;p_{1},...,p_{n}}^{n}\cap\Pi_{q;q_{1},...,q_{n}}^{n}\text{ implies
}T\in\Pi_{r;r_{1},...,r_{n}}^{n},
\]
provided that for each $j=1,...,n,$ one of the following conditions holds:

$\left(  i\right)  $ $X_{j}$ is an $\mathcal{L}_{\infty}$-space;

$\left(  ii\right)  $ $X_{j}$ is of cotype $2$ and $1\leq p_{j},q_{j}\leq2$;

$\left(  iii\right)  $ $X_{j}$ is of finite cotype $s_{j}>2$ and $1\leq
p_{j},q_{j}<s_{j}^{\ast}$;

$\left(  iv\right)  $ $p_{j}=q_{j}=r_{j}.$
\end{proposition}

Next lemma appears in \cite[Theorem $3.1$]{dd} without proof. We present a
proof for the sake of completeness:

\begin{lemma}
\label{lem}Let $s>0$. Suppose that $X_{j}$ has cotype $s_{j}$ for all
$j=1,...,n$ and at least one of the $s_{j}$ is finite. If%
\[
\frac{1}{s}\leq\frac{1}{s_{1}}+\ldots+\frac{1}{s_{n}},
\]
then
\[
\mathcal{L}\left(  X_{1},...,X_{n};Y\right)  =\Pi_{s;b_{1},...,b_{n}}%
^{n}\left(  X_{1},...,X_{n};Y\right)
\]
for
\[
b_{j}=1\text{ if }s_{j}<\infty\text{ and }b_{j}=\infty\text{ if }s_{j}%
=\infty.
\]

\end{lemma}

\begin{proof}
Let $j_{1},...,j_{k}\in\left\{  1,...,n\right\}  ,k\leq n$ such that
$s_{j_{1}},...,s_{j_{k}}$ are finite and $s_{j}=\infty$ if $j\not =%
j_{1},...,j_{k}.$

If $\left(  x_{i}^{\left(  j_{l}\right)  }\right)  _{i=1}^{\infty}\in
l_{1}^{w}(X_{j_{l}})$ and $(x_{i}^{\left(  j\right)  })_{i=1}^{\infty}\in
l_{\infty}(X_{j}),j\not =j_{l}$, $l=1,...,k,$ using Generalized H\"{o}lder
Inequality, we obtain%
\begin{align*}
\left(
{\textstyle\sum\limits_{i=1}^{\infty}}
\left\Vert T(x_{i}^{\left(  1\right)  },...,x_{i}^{\left(  n\right)
})\right\Vert ^{s}\right)  ^{\frac{1}{s}}  &  \leq\left\Vert T\right\Vert
\left(
{\textstyle\sum\limits_{i=1}^{\infty}}
\left(  \left\Vert x_{i}^{\left(  1\right)  }\right\Vert \cdots\left\Vert
x_{i}^{\left(  n\right)  }\right\Vert \right)  ^{s}\right)  ^{\frac{1}{s}}\\
&  \leq C\left\Vert T\right\Vert \left(
{\textstyle\sum\limits_{i=1}^{\infty}}
\left\Vert x_{i}^{\left(  j_{1}\right)  }\right\Vert ^{s_{j_{1}}}\right)
^{1/s_{j_{1}}}\cdots\left(
{\textstyle\sum\limits_{i=1}^{\infty}}
\left(  \left\Vert x_{i}^{\left(  j_{k}\right)  }\right\Vert \right)
^{s_{j_{k}}}\right)  ^{1/s_{j_{k}}}%
\end{align*}
where $C$ is such that
\[
\prod_{j=1,j\not =j_{1},...,j_{n}}^{n}\left\Vert x_{i}^{\left(  j\right)
}\right\Vert \leq C
\]
for all $i.$ Since $X_{j}$ has cotype $s_{j},$ for $j_{1},...,j_{k},$ we have%
\begin{align*}
\left(
{\textstyle\sum\limits_{i=1}^{\infty}}
\left\Vert T(x_{i}^{\left(  1\right)  },...,x_{i}^{\left(  n\right)
})\right\Vert ^{s}\right)  ^{\frac{1}{s}}  &  \leq C\left\Vert T\right\Vert
\left(
{\textstyle\sum\limits_{i=1}^{\infty}}
\left\Vert x_{i}^{\left(  j_{1}\right)  }\right\Vert ^{s_{j_{1}}}\right)
^{1/s_{j_{1}}}\cdots\left(
{\textstyle\sum\limits_{i=1}^{\infty}}
\left(  \left\Vert x_{i}^{\left(  j_{k}\right)  }\right\Vert \right)
^{s_{j_{k}}}\right)  ^{1/s_{j_{k}}}\\
&  =C\left\Vert T\right\Vert \prod_{t=1}^{k}\left(
{\textstyle\sum\limits_{i=1}^{\infty}}
\left\Vert id_{Xj_{_{t}}}\left(  x_{i}^{\left(  j_{t}\right)  }\right)
\right\Vert ^{s_{j_{t}}}\right)  ^{1/s_{j_{t}}}<\infty
\end{align*}
and the result follows.
\end{proof}

The main result of this note is the following Theorem. At first glance it
seems to have too restrictive assumptions, but Corollary \ref{C} and Example
\ref{E} will illustrate its usefulness:

\begin{theorem}
\label{teo}Let $s\geq1$ and $n,k$ be positive integers with $k\leq n$. If
$X_{j}$ has finite cotype $s_{j}\geq2$ for $j=1,...,k$, then
\[
\Pi_{p;p_{1},...,p_{k},q,...,q}^{n}\left(  X_{1},...,X_{n};Y\right)
\subseteqq\Pi_{r;r_{1},...,r_{k},q,...,q}^{n}\left(  X_{1},...,X_{n};Y\right)
\]
for any $(q,\theta)\in\lbrack1,\infty]\times(0,1),$%
\begin{align*}
1 &  \leq p_{j}\leq2\text{ (when }s_{j}=2\text{)},\\
1 &  \leq p_{j}<s_{j}^{\ast}\text{ (when }s_{j}>2\text{)}%
\end{align*}
and%
\begin{align*}
\text{ }\frac{1}{s} &  \leq\frac{1}{s_{1}}+\cdots+\frac{1}{s_{k}},\\
\frac{1}{r} &  =\frac{1-\theta}{s}+\frac{\theta}{p},\\
\frac{1}{r_{j}} &  =\frac{1-\theta}{1}+\frac{\theta}{p_{j}},
\end{align*}
for all $j=1,...,k$.
\end{theorem}

\begin{proof}
Let $T\in\Pi_{p;p_{1},...,p_{k},q,...,q}^{n}\left(  X_{1},...,X_{n};Y\right)
.$ By the previous lemma,%
\[
\mathcal{L}\left(  X_{1},...,X_{n};Y\right)  =\Pi_{s;1,...,1,\infty
,...,\infty}^{n}\left(  X_{1},...,X_{n};Y\right)  ,
\]
where $1$ is repeated $k$ times. A fortiori, we have%
\[
\mathcal{L}\left(  X_{1},...,X_{n};Y\right)  =\Pi_{s;1,...,1,q,...,q}%
^{n}\left(  X_{1},...,X_{n};Y\right)  .
\]
So,%
\[
T\in\Pi_{s;1,...,1,q,...,q}^{n}\left(  X_{1},...,X_{n};Y\right)  \cap
\Pi_{p;p_{1},...,p_{k},q,...,q}^{n}\left(  X_{1},...,X_{n};Y\right)  .
\]

From Proposition \ref{prop} we get
\[
T\in\Pi_{r;r_{1},...,r_{k},q,...,q}^{n}\left(  X_{1},...,X_{n};Y\right)  .
\]

\end{proof}

\begin{corollary}
\label{C}Let $k$ be a natural number, $k\geq2$ and $q\in\lbrack1,\infty).$ If
$X_{j}$ has finite cotype $s_{j}\geq2,$ $j=1,...,k$ and $1\leq1/s_{1}%
+\cdots+1/s_{k},$ then%
\[
\Pi_{p;p_{1},...,p_{k},q,...,q}^{n}\left(  X_{1},...,X_{n};Y\right)
\subseteqq\Pi_{r;r_{1},...,r_{k},q,...,q}^{n}\left(  X_{1},...,X_{n};Y\right)
,
\]
where $p_{j}=p$ and $r_{j}=r$ for all $j=1,...,k$, for all $r$ so that
\begin{align*}
1 &  \leq r<p<\min s_{j}^{\ast}\text{ if }s_{j}\neq2\text{ for some
}j=1,...,k,\\
1 &  \leq r<p\leq2\text{ if }s_{j}=2\text{ for all }j=1,...,k.
\end{align*}
In particular%
\[
\Pi_{p;p}^{n}\left(  X_{1},...,X_{n};Y\right)  \subseteqq\Pi_{r;r}^{n}\left(
X_{1},...,X_{n};Y\right)
\]
for all $r$ so that
\begin{align*}
1 &  \leq r<p<\min s_{j}^{\ast}\text{ if }s_{j}\neq2\text{ for some
}j=1,...,k,\\
1 &  \leq r<p\leq2\text{ if }s_{j}=2\text{ for all }j=1,...,k.
\end{align*}

\end{corollary}

\begin{proof}
Since $1\leq1/s_{1}+\cdots+1/s_{k},$ we can use $s=1$ in the previous theorem.
Since $p=p_{i}$ and $r=r_{i}$ for all $i=1,...,k$ and $s=1,$ we conclude that%
\[
\Pi_{p;p,...,p,q,...,q}^{n}\left(  X_{1},...,X_{n};Y\right)  \subseteqq
\Pi_{r;r,...,r,q,...,q}^{n}\left(  X_{1},...,X_{n};Y\right)  .
\]
In fact, for any $1\leq r<p$ there is a $\theta\in\lbrack0,1]$ so that%
\[
\frac{1}{r}=\frac{1-\theta}{1}+\frac{\theta}{p}%
\]
and since $p=p_{i}$ and $r=r_{i}$, the same $\theta\in\lbrack0,1]$ satisfies%
\[
\frac{1}{r_{i}}=\frac{1-\theta}{1}+\frac{\theta}{p_{i}}.
\]

Choosing $q=p,$ since $r<p$ we have
\[
\Pi_{p;p}^{n}\left(  X_{1},...,X_{n};Y\right)  \subseteqq\Pi
_{r;r,...,r,p,...,p}^{n}\left(  X_{1},...,X_{n};Y\right)  \subseteqq\Pi
_{r;r}^{n}\left(  X_{1},...,X_{n};Y\right)  .
\]
\bigskip
\end{proof}

\begin{example}
\label{E}Let $X_{4},...,X_{n},Y$ be arbitrary Banach spaces. Then%
\[
\Pi_{p;p,p,p,q,...,q}^{n}\left(  \ell_{3},\ell_{3},\ell_{3},X_{4}%
,...,X_{n};Y\right)  \subseteqq\Pi_{r;r,r,r,q,...,q}^{n}\left(  \ell_{3}%
,\ell_{3},\ell_{3},X_{4},,,.,X_{n};Y\right)
\]
for all $q\in\lbrack1,\infty)$ and $1\leq r<p<3^{\ast}.$ In particular%
\[
\Pi_{p;p}^{n}\left(  \ell_{3},\ell_{3},\ell_{3},X_{4},...,X_{n};Y\right)
\subseteqq\Pi_{r;r}^{n}\left(  \ell_{3},\ell_{3},\ell_{3},X_{4},,,.,X_{n}%
;Y\right)
\]
for all $1\leq r<p<3^{\ast}.$
\end{example}

\end{document}